\documentclass[11pt]{article}
\usepackage{amsthm, amsmath, amssymb, amsfonts, url, booktabs, tikz, setspace, fancyhdr, bm}
\usepackage{geometry}
\usepackage{hyperref, enumerate}
\usepackage[shortlabels]{enumitem}
\usepackage[babel]{microtype}
\usepackage[english]{babel}
\usepackage[capitalise]{cleveref}
\usepackage{comment}
\usepackage{bbm}
\usepackage{csquotes}
\usepackage{hyphenat}
\usepackage{mathabx}
\usepackage{tikz}
\usepackage{graphicx}
\usepackage{float}
\usepackage{xcolor}
\usepackage{mathtools}
\usetikzlibrary{positioning, arrows.meta, shapes.geometric}
\usepackage{amsmath}

\counterwithin{figure}{section}

\newtheorem{theorem}{Theorem}[section]

\newtheorem{lemma}[theorem]{Lemma}
\newtheorem{problem}[theorem]{Problem}
\newtheorem{cor}[theorem]{Corollary}
\newtheorem{claim}[theorem]{Claim}

\newtheorem{case}{Case}

\usetikzlibrary{decorations.pathmorphing}

\theoremstyle{definition}

\newtheorem*{defn-non}{Definition}

\newlist{Case}{enumerate}{2}
\setlist[Case, 1]{%
    label           =   {\bfseries Case \arabic*.},
    labelindent=1em ,labelwidth=1.3cm, labelsep*=1em, leftmargin =!
}
\setlist[Case, 2]{%
    label           =   {\bfseries Subcase \arabic{Casei}.\arabic*.},
    labelindent=-1em ,labelwidth=1.3cm, labelsep*=1em, leftmargin =!
}

\usepackage{todonotes}

\begin{document}
%%%%%%%%%%%%%%%%%%%%%%%%%%%%%%%%%%%%%%%%%%%%%%%%%%%%%%%
\title{\bf\Large  On the $(k+2,k)$-problem of Brown, Erd\H{o}s and S\'{o}s \\ for even integers $k$ }

\author{%
   Yan Wang\textsuperscript{\textdagger},%
    \quad Jiasheng Zeng\textsuperscript{\S}%
}
\footnotetext[1]{\scriptsize
\noindent\textsuperscript{\textdagger}{School of Mathematical Sciences, Shanghai Jiao Tong University, Shanghai 200240, China. Email: yan.w@sjtu.edu.cn }}
\footnotetext[2]{\scriptsize
\noindent\textsuperscript{\S}{School of Mathematical Sciences, Shanghai Jiao Tong University, Shanghai 200240, China. Email: jasonzeng@sjtu.edu.cn}  }

\date{\today}
%%%%%%%%
\maketitle
%\footnote{footnote}
%%%%%%%%%%%%%%%%%%%%%%%%%%%%%%%%%%%%%%%%%%%%%%%%%
\begin{abstract}
Let $f^{(r)}(n;s,k)$ denote the maximum number of edges in an $r$-graph on $n$ vertices in which every $k$ edges span more than $s$ vertices. Brown, Erd\H{o}s and S\'{o}s in 1973 conjectured that for every $k\geq 2$, the limit $\lim_{n\to\infty} n^{-2} f^{(3)}(n;k+2,k)$ exists and verified the conjecture for $k=2$ by showing that $\lim_{n\to\infty} n^{-2} f^{(3)}(n;4,2)=\frac{1}{6}$. Delcourt and Postle, building on the work of Glock, Joos, Kim, Kühn, Lichev and Pikhurko, proved that for every $k\geq 2$, the limit $\lim_{n\to\infty} n^{-2} f^{(3)}(n;k+2,k)$ exists, thereby solving this conjecture. Their approach was later generalised by Shangguan to every uniformity $r\geq 4$: the limit $\lim_{n\to\infty} n^{-2} f^{(r)}(n; rk-2k+2,k)$ exists for all $r\geq 3$ and $k\geq 2$. However, its exact value was not determined.
%Let $f^{(r)}(n;s,k)$ denote the maximum number of edges in an $r$-graph on $n$ vertices in which every $k$ edges span more than $s$ vertices. Delcourt and Postle, building on the work of Glock, Joos, Kim, Kühn, Lichev and Pikhurko, proved that for every $k\geq 2$, the limit $\lim_{n\to\infty} n^{-2} f^{(3)}(n;k+2,k)$ exists, thereby solving a conjecture of Brown, Erd\H{o}s and S\'{o}s from 1973. Their approach was later generalised by Shangguan to every uniformity $r\geq 4$: the limit $\lim_{n\to\infty} n^{-2} f^{(r)}(n; rk-2k+2,k)$ exists for all $r\geq 3$ and $k\geq 2$, though its exact value was not determined.

When $k\in\{2,3,\ldots,7\}$, the exact values of $\lim_{n\to\infty} n^{-2} f^{(r)}(n; rk-2k+2,k)$  were determined by Glock, Joos, Kim, K\"{u}hn, Lichev, Pikhurko, R\"{o}dl and Sun. Very recently, the limit for $k=8$ and $r\geq 4$ was determined by Pikhurko and Sun. For a general even integer $k$, Letzter and Sgueglia obtained the exact values of $\lim_{n\to\infty} n^{-2} f^{(r)}(n;rk-2k+2,k)$ for every even integer $k$ and uniformity $r\geq 2+\sqrt{2}\,k^{3/2}$. In this paper, we determine the exact value of $\lim_{n\to\infty} n^{-2} f^{(r)}(n;rk-2k+2,k)$ for every even integer $k\geq 4$ and $r\geq 2+\sqrt{\frac{3}{2}k-4}$, and show that it is $\frac{1}{r^2-r}.$
\end{abstract}

%%%%%%%%%%%%%%%%%%%%%%%%%%%%%%%%%%%%%%%%%%%%%%%%%
%\section{\texorpdfstring{$((r-2)k+3,k)$}{}-type, i.e., BESC}
\section{Introduction}

For a given $r \ge 2$, an \emph{$r$-uniform hypergraph} (or \emph{$r$-graph} for short) $H$ consists of a vertex set $V(H)$ and an edge set $E(H) \subseteq \binom{V(H)}{r}$. We denote by $|V(H)|$ and $|H|$ the number of vertices and edges of $H$, respectively. Given an $r$-graph $G$ and a family of $r$-graphs $\mathcal{F}$, we say that $G$ is \emph{$\mathcal{F}$-free} if it contains no member of $\mathcal{F}$ as a subhypergraph. The \emph{Tur\'{a}n number} $\mathrm{ex}_r(n,\mathcal{F})$ is defined as the maximum number of edges in an $\mathcal{F}$-free $r$-graph on $n$ vertices. In this paper, we consider the family $$\mathcal{F}^{(r)}(s,k) := \{ r\text{-graph } H : |H| = k \text{ and } |V(H)| \le s \}. $$
Brown, Erd\H{o}s, and S\'os~\cite{brown1973some-BES-original} were the first to systematically study the function$$ f^{(r)}(n;s,k) := \mathrm{ex}_r(n,\mathcal{F}^{(r)}(s,k)). $$ They showed that \begin{equation}\label{most general fr n s k}
    \Omega(n^{(rk-s)/(k-1)})=f^{(r)}(n;s,k)=O(n^{\lceil(rk-s)/(k-1)\rceil}).
\end{equation}

A central and notoriously difficult case arises when $r=3$ and $s=k+3$. It follows from (\ref{most general fr n s k}) that $$\Omega(n^{\frac{2k-3}{k-1}})=f^{(3)}(n;k+3,k)=O(n^2).$$ In this setting, Brown, Erd\H{o}s, and S\'os (see~\cite{Erdos_Frankl_Rodl_1986}) conjectured that $f^{(3)}(n;k+3,k)=o(n^2),$ which is now known as the $(k+3,k)$-conjecture. Despite extensive efforts, this conjecture remains widely open, having only been resolved for $k=3$ by Ruzsa and Szemer\'edi in their celebrated $(6,3)$-theorem~\cite{Ruzsa_Szemeredi_1978}. This result is of fundamental importance: its proof led to the development of the triangle removal lemma (see~\cite{Conlon_Fox_2013_survey} for a survey), constituted one of the earliest applications of Szemer\'edi's regularity lemma, and implies Roth's theorem~\cite{Roth_1953_3ap} on $3$-term arithmetic progressions.

The difficulty of the $(k+3,k)$-conjecture also motivated major advances in extremal combinatorics, most notably the hypergraph removal lemma due to Gowers~\cite{Gowers_2007}, Nagle, R\"odl and Schacht~\cite{Nagle_Rodl_Schacht_2006}, and R\"odl and Skokan~\cite{Rodl_Skokan_2004_regular,Rodl_Skokan_2006_apply}. Given the current gap in understanding, it is natural to consider approximate versions: what is the smallest integer $d= d(k)$ such that $f^{(3)}(n; k + d,k) = o(n^2)$? In this direction, S\'ark\"ozy and Selkow~\cite{Sarkozy_Selkow_2004_k+d} showed that \begin{equation}\label{k+d on2}
    f^{(3)}(n; k + 2 + \lfloor \log_2 k \rfloor, k)=o(n^2),
\end{equation} and this was later improved by Conlon, Gishboliner, Levanzov, and Shapira~\cite{CONLON20231JCTB} to \begin{equation}f^{(3)}(n; k + O(\tfrac{\log k}{\log\log k}), k)=o(n^2).\end{equation} For small values of $k$, the only progress since~\cite{Sarkozy_Selkow_2004_k+d} was made by Solymosi and Solymosi~\cite{Solymosi_Solymosi_2017_small_k}, who improved the bound $f^{(3)}(n;15,10) = o(n^2) $ following from (\ref{k+d on2}) to $f^{(3)}(n;14,10) = o(n^2).$ Apart from this line of research, Shapira and Tyomkyn~\cite{shapira2021ramsey} also considered a Ramsey-type $(k+3,k)$ problem. Moreover, the $(k+3,k)$-conjecture has been extended to a more general case (see~\cite{Alon_Shapira_2006}), which states that for every $2 \leq \ell < r$ and $k \geq 3$, we have $f^{(r)}(n; (r-\ell)k + \ell + 1, k) = o(n^{\ell})$.

In addition to the $(k+3,k)$-conjecture, the cases when the magnitude of $f^{(r)}(n;s,k)$ has been given by (\ref{most general fr n s k}) have also been extensively studied.
When $t := (rk-s)/(k-1)$ is an integer, it holds that $f^{(r)}(n; (r-t)k+t, k) = \Theta(n^t).$ We are particularly interested in the case when $t=2$, i.e., $s=(r-2)k+2$. In this case, $f^{(r)}(n; (r-2)k+2, k) = \Theta(n^2).$ A natural question is whether the limit $$\pi(r,k) := \lim_{n\to\infty} n^{-2} f^{(r)}(n; (r-2)k+2, k)$$ exists. This was conjectured by Brown, Erd\H{o}s, and S\'os~\cite{brown1973some-BES-original} for $r=3$. In particular, they verified the conjecture for $k=2$ and showed that $\pi(3,2)=\frac{1}{6}.$ Glock~\cite{glock2019triple-pi33-BLMS} later resolved the case when $k=3$ by proving that $\pi(3,3)=\frac{1}{5}$. More recently, Glock, Joos, Kim, K\"{u}hn, Lichev, and Pikhurko~\cite{glock20246-k4} established the case $k=4$, showing that $\pi(3,4)=\frac{7}{36}$. Motivated by these developments, Delcourt and Postle~\cite{delcourt2024limit-PAMS-Resloved} verified the Brown--Erd\H{o}s--S\'os conjecture by proving that the limit $\pi(3,k)$ exists for all $k \ge 2$, though the exact values were not determined.

For general case, Shangguan~\cite{shangguan2023degenerate} proved that $\pi(r,k)$ exists for all $r \ge 4$. A natural problem is to determine the exact value of this limit. For small values of $k$, R\"{o}dl's result on the existence of approximate Steiner systems~\cite{rodl1985packing} implies that $\pi(r,2)=\frac{1}{r^2-r}$ for all $r \ge 3$. Moreover, for $k=3$ and $k=4$, Glock, Joos, Kim, K\"{u}hn, Lichev, and Pikhurko~\cite{glock20246-k4} proved that $$\pi(r,3)=\frac{1}{r^2-r-1} \text{ for }r\geq 3, \quad \pi(r,4)=\frac{1}{r^2-r} \text{ for }r\geq 4, \text{ and } \pi(3,4)=\frac{7}{36}.$$

Recently, Glock, Kim, Lichev, Pikhurko, and Sun~\cite{Glock_Kim_Lichev_Pikhurko_Sun_2025} determined the values of $\pi(r,k)$ for $k \in \{5,6,7\}$. In particular, they showed that $$\pi(r,5)=\pi(r,7)=\frac{1}{r^2-r-1} \quad \text{for every } r \ge 3,$$
and
$$\pi(3,6)=\frac{61}{330} \quad \text{and} \quad \pi(r,6)=\frac{1}{r^2-r} \quad \text{for every } r \ge 4.$$
More recently, Pikhurko and Sun~\cite{pikhurko2025quadratic-8-edges} proved that for every $r \ge 4$, $\pi(r,8)=\frac{1}{r^2-r}$. Using the probabilistic method, they also showed that $\pi(3,8)\geq \frac{3}{16}$ and conjectured that this bound is tight. For even integer $k$, Letzter and Sgueglia~\cite{letzter2025problemSIAM-k1.5} showed that there exists $r_0(k)$ such that \begin{equation}\label{pi r k explicit value}
    \pi(r,k)=\frac{1}{r^2-r}, \quad \text{ for all }r\geq r_0(k).
\end{equation} Specifically, they proved the following result:
\begin{theorem}[\cite{letzter2025problemSIAM-k1.5}]
    For every even integer $k\geq 4$ and integer $r\geq \sqrt{2}k^{\frac{3}{2}}+2$, we have $\pi(r,k)=\frac{1}{r^2-r}$. 
\end{theorem}

They also asked for the smallest $r$, denoted as $r_0(k)$, such that the equation (\ref{pi r k explicit value}) holds. Then $r_0(2)=3$ and $r_0(4)=r_0(6)=r_0(8)=4$ according to the previous results. Moreover, Letzter and Sgueglia's result implies that $r_0(k)\leq \sqrt{2}k^{\frac{3}{2}}+2$ for every even integer $k$. In this paper, we show that $r_0(k)$ can be improved to $\sqrt{\frac{3}{2}k-4}+2$ for every even integer $k\geq 4$ as follows.

\begin{theorem}\label{main large r even k sqrtk}
    For every even integer $k\geq 4$ and integer $r\geq \sqrt{\frac{3}{2}k-4}+2$, we have $\pi(r,k)=\frac{1}{r^2-r}$.
\end{theorem}

\section{Preliminaries}

In this paper, we mainly follow the notation used in~\cite{Glock_Kim_Lichev_Pikhurko_Sun_2025} and~\cite{pikhurko2025quadratic-8-edges}. For a positive integer $m$, we write $[m]$ for the set $\{1,2,\dots,m\}$. For a set $X$, let $\binom{X}{m}$ denote the family of all $m$-subsets of $X$. We often write $xy$ to denote the unordered set $\{x,y\}$.  We identify an $r$-graph $G$ with its edge set. That is, $|G|=|E(G)|$ and $V(G)=\bigcup_{e\in E(G)} e$. For two $r$-graphs $G$ and $H$, we define their union $G\cup H$ by $E(G\cup H)=E(G)\cup E(H)$ and their difference $G\setminus H$ by $E(G\setminus H)=E(G)\setminus E(H)$.

We call an $r$-graph a \emph{diamond} if it has exactly two edges that intersect in exactly two vertices. For positive integers $s$ and $k$, an $(s,k)$-configuration is an element of $\mathcal{F}^{(r)}(s,k)$, that is, an $r$-graph with $k$ edges and at most $s$ vertices. In particular, when $s=(r-2)k+2$ and $s=(r-2)k+1$, we abbreviate these as a $k$-configuration and a $k^{-}$-configuration, respectively. Let $\mathcal{G}_{k}^{r}$ denote the family of all $k$-configurations and all $\mathcal{\ell}^{-}$-configurations with $\ell\in [2,k-1]$, namely, $$\mathcal{G}_{k}^{(r)}=\mathcal{F}^{(r)}((r-2)k+2,k)\cup \left(\bigcup_{\ell=2}^{k-1}\mathcal{F}^{(r)}((r-2)\ell+1,\ell)\right).$$ In the following sections, we will see that $\mathcal{G}_k^{(r)}$ is closely related to $\pi(r,k)$.

Given an $r$-graph $G$, a pair of distinct vertices $xy$ (not necessarily in $\binom{V(G)}{2}$), and a set $A\subseteq \mathbb{N}\cup\{0\}$, we say that $G$ \emph{$A$-claims} the pair $xy$ if for every $i\in A$ there exist $i$ distinct edges $e_1,\dots,e_i$ such that $$ |\{x,y\}\cup e_1\cup\cdots\cup e_i|\le (r-2)i+2. $$ Let $P_A(G)$ denote the set of all vertex pairs $A$-claimed by $G$. Note that $0\in P_A(G)$ always holds, and hence $P_A(G)$ is never empty. In particular, when $A=\{i\}$, we write \emph{$i$-claims} (and respectively $P_i(G)$) instead of $\{i\}$-claims (and respectively $P_{\{i\}}(G)$). Let $C_{G}(xy)$ be the set of nonnegative integers $i$ such that the pair $xy$ is $i$-claimed by $G$, namely, $$C_{G}(xy)=\{i\geq 0: \exists \text{ distince }e_1\cdots,e_i \in E(G)\text{ such that }|\{x,y\}\cup(\cup_{j=1}^ie_j)|\leq (r-2)i+2\}.$$
Moreover, for distinct $A,B\subseteq \mathbb{N}\cup \{0\}$, we say $G$ $\overline{A}B$-claims a pair $xy$ if $A\cap C_{G}(xy)= \emptyset$ and $B\cap C_{G}(xy)\neq \emptyset$. Inparticular, when $A=\{i\}$ and $B=\{j\}$ for $i\neq 
j$, we simply omit the curly brackets. Also, we let $P_{\overline{1}i}(G):=P_i(G)\setminus P_{1}(G)$ be the set of pairs in $\binom{V(G)}{2}$ that are $\overline{1}i$-claimed by $G$.

\section{Proof of Theorem~\ref{main large r even k sqrtk}}
\subsection{Lower bound}

We will need the following result to prove the lower bound for $\pi(r,k)$ in Theorem~\ref{main large r even k sqrtk}.

\begin{lemma}[\cite{glock20246-k4},Theorem 3.1]\label{lower bound lemma} Fix $k\geq 2$ and $r\geq 3$. Let $F$ be a $\mathcal{G}_k^{(r)}$-free $r$-graph. Then $$\liminf_{n\to \infty}\frac{f^{(r)}(n;(r-2)k+2,k)}{n^2}\geq \frac{|F|}{2|P_{\leq\lfloor\frac{k}{2}\rfloor}(F)|},$$ where we define $$P_{\leq t}(F):=\{xy\in \binom{V(F)}{2}:C_{F}\cap [t]\neq \emptyset\}$$ to consist of all pairs $xy$ of $F$ such that $C_{F}(xy)$ contains some $i$ with $1\leq i\leq t$.
    
\end{lemma}
\begin{proof}[Proof of the lower bound of Theorem~\ref{main large r even k sqrtk}] The lower bound $\pi(r,k)\geq \frac{1}{r^2-r}$ follows from Lemma~\ref{lower bound lemma} with the $r$-graph $F$ being a single edge.
\end{proof}

\subsection{Upper bound}

We will use the following lemma to establish the upper bound. This lemma was proved by Delcourt and Postle~\cite{delcourt2024limit-PAMS-Resloved} for $r=3$ and by Shangguan~\cite{shangguan2023degenerate} for $r \ge 4$. Consequently, to obtain the upper bound for $\pi(r,k)$, it suffices to consider $\mathcal{G}_k^{(r)}$-free $r$-graphs. 

\begin{lemma}\label{suffice upperbound}
    For all fixed $r\geq 3$ and $k\geq 3$, $$\limsup_{n\to \infty}n^{-2}f^{r}(n;(r-2)k+2,k)\leq \limsup_{n\to\infty}n^{-2}\mathrm{ex}(n,\mathcal{G}_k^{(r)}).$$
\end{lemma}

Our proof is inspired by~\cite{glock20246-k4,Glock_Kim_Lichev_Pikhurko_Sun_2025,pikhurko2025quadratic-8-edges}. We now give an outline of the proof. Assume that $G$ is an $n$-vertex $\mathcal{G}_k^{(r)}$-free $r$-graph. We begin with the trivial partition of $E(G)$ into single edges and repeatedly apply certain merging rules to combine parts into larger clusters. The $\mathcal{G}_k^{(r)}$-freeness restricts the possible structures of the resulting clusters. For each final cluster (which forms a subgraph of $G$), we assign weights to certain vertex pairs inside the cluster. This assignment is designed so that every pair of vertices in $V(G)$ receives total weight at most $1$. By comparing the number of edges in a cluster with the total weight assigned within it, we obtain an upper bound on $|G|$ by bounding the ratio between the number of edges and the total weight for each possible cluster.

%\subsection{Merging and analysing}
\subsubsection{Mergeability}

Let $G$ be an $r$-graph and $\mathcal{P}$ be a partition of its edge set $E(G)$. We view each element of $\mathcal{P}$ as a subhypergraph of $G$. Let $A,B\subseteq \mathbb{N}\cup \{0\}$ (not necessarily different). For edge-disjoint $F,H\subseteq G$ and a pair $xy$, we say $F$ and $G$ is $(A|B)$-mergeable (via $xy$) if $A\subseteq C_{F}(xy)$ and $B\subseteq C_{H}(xy)$. We shortly write $A|B$-mergeable to denote $(A|B)$-mergeable or $(B|A)$-mergeable. An $A|B$-merging of $\mathcal{P}$, denoted by $\mathcal{M}_{A|B}(\mathcal{P})$, is a partition of $E(G)$ obtaining from $\mathcal{P}$ by iteratively merging a pair of distinct $A|B$-mergeable parts and replacing them by their union as much as possible. A partial $A|B$-cluster $F$ is a subgraph of $G$ which can appear as a part in some intermediate stage of the $A|B$-merging process starting with $\mathcal{P}$. Let $\mathcal{M}'_{A|B}(\mathcal{P})$ denote the set of all partial $A|B$-clusters. In particular, when $A=\{1\}$ and $B=\{j\}$, we write $(j)$ and $j$ instead of $(\{1\}|\{j\})$ and $\{1\}|\{j\}$ for short.

Given an $r$-graph $G$, let $$\mathcal{M}_1=\mathcal{M}_{1|1}(\mathcal{P}_{trivial})$$ be the $1$-merging of the trivial partition $\mathcal{P}_{trivial}$ of $G$ into single edges. Note that the elements of $\mathcal{M}_1$ admit the following equivalent description. We call $F \subseteq G$ \emph{connected} if for any $X,Y \in F$ there exists a sequence of edges $X_1=X, X_2, \ldots, X_m=Y$ in $F$ such that $|X_i \cap X_{i+1}| \ge 2$ for every $i \in [m-1]$. Thus, the $1$-clusters are precisely the maximal connected subgraphs of $G$. We will need the following. 

%\begin{lemma}\label{lemma joint remain k-con}
%Let $k \geq 4$, and let $G$ be an $r$-uniform hypergraph that is $\mathcal{G}_k^{(r)}$-free. Suppose $A, B \subseteq E(G)$ satisfy $|A| = a$, $|B| = b$, and $A \cap B = \emptyset$. Assume that $|V(A)| \leq (r-2)a + 2$ and $|V(B)| \leq (r-2)b + 2$. If, for some $2 \leq t \leq k-1$, the sets $A$ and $B$ are $1|t$-mergeable, then $$ |V(A) \cup V(B)| \leq (r-2)(a+b) + 2,$$and equality can occur only when $|V(A) \cap V(B)| = 2$.    
%\end{lemma}

\begin{lemma}[Trimming Lemma~\cite{Glock_Kim_Lichev_Pikhurko_Sun_2025}]\label{trimming lemma} Fix an $r$-graph $G$, a partition $\mathcal{P}$ of $E(G)$ and sets $A,B\subseteq \mathbb{N}$. Suppose that, for all $(A|B)$-mergeable (and thus edge-disjoint) $F,H\in \mathcal{M}'_{A|B}(\mathcal{P})$, there exist $(A|B)$-mergeable $F',H'\in \mathcal{P}$ such that $F'\subseteq F$ and $H'\subseteq H$. 

Then for every partial $A|B$-clusters $F_0\subseteq F$, there is an ordering $F_1,\cdots,F_s$ of the elements of $\mathcal{P}$ that lie inside $F\setminus F_0$ such that, for every $i\in [s]$, $\bigcup_{j=0}^{i-1}F_j$ and $F_i$ are $A|B$-mergeable (and, in particular, $\bigcup_{j=0}^{i-1}F_j$ is a partial $A|B$-cluster for every $i\in [s]$).
\end{lemma}

It is easy to see that the assumption of Lemma~\ref{trimming lemma} is satisfied when $A=B=\{1\}$. We will frequently use the case when $A=B=\{1\}$, and therefore state it separately.

\begin{cor}\label{connect trimming lemma}
    For every pair $F_0\subseteq F$ of connected $r$-graphs, there is an ordering $X_1,\cdots,X_s$ of the edges in $F\setminus F_0$ such that for every $i\in [s]$, the $r$-graph $F_0\cup \{X_1,\cdots,X_i\}$ is connected.
\end{cor}

%Moreover, if $G$ is $\mathcal{G}_k^{(r)}$-free, then we have the following.

%\begin{lemma}\label{number vts edge M1}
%    With the above notation, if $G$ is $\mathcal{G}_{k}^{(r)}$-free, then every $F\in \mathcal{M}_1$ satisfies that $|F|=i\leq k-1$ and \begin{equation}\label{value M1 P1 and P12}
%        |P_1(F)|=i\binom{r}{2}-i+1 \text{ and }|P_{\overline{1}2}(F)|\geq (i-1)(r-2)^2.
%    \end{equation}
%\end{lemma}

%\begin{proof}
%    TODO
%\end{proof}

\subsubsection{Structural results for \texorpdfstring{$\mathcal{M}_2$}{}}

In the sequel, we always assume that $G$ is an $r$-uniform hypergraph on $n$ vertices that is $\mathcal{G}_k^{(r)}$-free. We consider the partition$$\mathcal{M}_2 := \mathcal{M}_{1|2}(\mathcal{M}_1),$$ which is obtained from $1$-clusters $\mathcal{M}_1$ by iteratively merging $(2)$-mergeable pairs. Let $\mathcal{M}'_2:=\mathcal{M}'_{1|2}(\mathcal{M}_1)$ be the set of all partial $1|2$-clusters. It is easy to see that for any $F \in \mathcal{M}'_2$, we have $|F| \neq k$; otherwise, $F$ would be a $k$-configuration. For $F\in \mathcal{M}'_2$ which is merged from $1$-clusters $F_1,\cdots,F_m$ in this order as in Lemma~\ref{trimming lemma}, the sequence of sizes $(|F_1|,\cdots,|F_m|)$ is called a \textit{composition} of $F$. Note that when $F \in \mathcal{M}'_2$, the value of $m$ is independent of the order of the merging process. Therefore, we call it the \emph{merging number} of $F$ and denote it by $m(F)$. Given $F\in \mathcal{M}'_2$, we say that a diamond $D\subseteq F$ is \textit{flexible} (in $F$) if for each $e\in D$, there is exactly $1$ vertex in $V(e)\cap V(F\setminus D)$. Therefore, if $D\subseteq F$ is a flexible diamond, then $|V(D)\cap V(F\setminus D)| = 2$.

The following statement is a consequence of Lemma~\ref{trimming lemma}.

\begin{lemma}\label{cor 1 of trimming lemma}
%Fix an $r$-graph $G$ and let $\mathcal{P}_{\mathrm{trivial}}$ be the trivial partition of $E(G)$. Let $\mathcal{M}_1=\mathcal{M}_{1|1}(\mathcal{P}_{\mathrm{trivial}}), \mathcal{M}_2=\mathcal{M}_{1|2}(\mathcal{M}_1).$
For any $A_0 \in \mathcal{M}'_{1|2}(\mathcal{M}_1)$, suppose that a diamond $D \subseteq A_0$ is flexible. Let $T \in \mathcal{M}_1$ such that $T$ and $A_0$ are $2|1$-mergeable via $f \in T$ and $\{e_1,e_2\} \subseteq A_0$. If $\{e_1,e_2\} \cap D=\emptyset$, then there exists an ordering $T=T_0,T_1,\ldots,T_s$ such that for every $i \in [s]$, the set $\bigcup_{j=0}^{i-1} T_j$ and $T_i$ are $1|2$-mergeable, and $\bigcup_{i=0}^{s} T_i = A_0 \cup T \setminus D.$
\end{lemma}

\begin{figure}[htbp]
\centering
\includegraphics[width=0.2\textwidth]{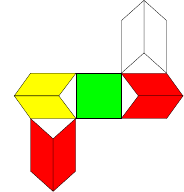}
\caption{This is an example of a $4$-uniform hypergraph with Property $\mathbb{P}$ in the case when $k=6$ and $r=4$. Each quadrilateral represents a hyperedge. If the green edges are taken as $T$, then the two red diamonds can serve as two flexible diamonds $T_1$ and $T_{\ell}$ in the definition.}
\label{fig:hypergraph}
\end{figure}

Let $F \in \mathcal{M}'_2$. We say that $F$ satisfies \emph{Property $\mathbb{P}$} (see Figure~\ref{fig:hypergraph} as an example) if for every $T \in F$ there exist a positive integer $\ell$ and $\ell$ $1$-clusters $T_1,T_2,\ldots,T_\ell$ merging in this order such that
\begin{equation}\label{Property P}
H=\bigcup_{p=1}^{\ell} T_p \subseteq F, 
\qquad 
\left|\bigcup_{p=1}^{\ell} T_p\right| = k+1,
\qquad 
T \in \{T_1,T_2,\ldots,T_\ell\}
\end{equation}
and $T_1$ and $T_\ell$ are two flexible diamonds in $H$ that $\overline{1}2$-claim a pair $1$-claimed by $T_2$ and $T_{\ell-1}$, respectively. Therefore, 
\begin{equation}\label{Property P 2}
    |V(T_1)\cap V(\cup_{i=2}^{\ell}T_p)|=2, \qquad |V(T_\ell) \cap V(\cup_{i=1}^{\ell-1}T_{p})|=2.
\end{equation}

\begin{theorem}\label{we can add 2 edges each time}
Suppose $F \in \mathcal{M}'_2$ satisfies Property $\mathbb{P}$. Then for any $S \in \mathcal{M}_1$ such that $F \cup S \in \mathcal{M}'_2$, we have $|S|=2$, and $F \cup S$ also satisfies Property $\mathbb{P}$.
\end{theorem}
\begin{proof}
Without loss of generality, suppose that $S$ and $T \in F$ are $1|2$-mergeable. Since $F$ satisfies Property $\mathbb{P}$, there exists $$ F_0 = T_1 \cup \cdots \cup T_m \subseteq F$$ by definition and thus satisfying~(\ref{Property P}) and~(\ref{Property P 2}). In particular, $D_0 = T_1$ and $D'_0 = T_m$ are diamonds. Let $A_0 = F_0 \setminus D_0$ and $A'_0 = F_0 \setminus D'_0$. Then $|A_0| = |A'_0| = k-1$ and $$|V(A_0)| = |V(A'_0)| = (r-2)(k-1) + 2.$$ We further prove that $S$ and $T$ are $(2|1)$-mergeable. Otherwise, there exist $\{x,y\} \subseteq e \in S$ and $f_1,f_2 \in T$ such that $|f_1 \cap f_2| = 2$ and $\{x,y\} \subseteq f_1 \cup f_2$. Since $|f_1 \cap f_2| = 2$, we have 
$$
\{f_1,f_2\} \cap D_0 = \emptyset 
\quad \text{or} \quad 
\{f_1,f_2\} \cap D'_0 = \emptyset.
$$
Without loss of generality, assume $\{f_1,f_2\} \cap D_0 = \emptyset$. Then $|A_0 \cup \{e\}| = k$ and $$ |V(A_0) \cup V(e)| \le (r-2)k + 2, $$ which contradicts the assumption that $G$ is $\mathcal{G}_k^{(r)}$-free. 

Thus, assume that there exist $e_1,e_2\in S$, $f\in F_0$, and $\{x,y\}\subseteq f$ such that $$|\{x,y\}\cup e_1\cup e_2|\leq 2r-2.$$ Without loss of generality, assume that $f\notin D_0$. Then $S$ and $A_0$ are $(2,1)$-mergeable. Let $F'_0=A_0\cup S$. By Lemma~\ref{cor 1 of trimming lemma}, there exist $T_1^0,T_2^0,\ldots,T_\ell^0\in F'_0$ that can merge in this order and form $F'_0$, where $T_1^0=S$ and $$\left|\bigcup_{j=2}^{\ell} T_j^0\right|=k-1 .$$ Let $p\in[\ell]$ be the smallest index such that $$\left|T_1^0\cup T_2^0\cup\cdots\cup T_p^0\right|\ge k+1 .$$ Then $$ \left|T_1^0\cup\cdots\cup T_{p-1}^0\right|\le k-1 \quad\text{and}\quad |T_p^0|\ge 2. $$
Let $H=T_1^0\cup\cdots\cup T_{p-1}^0$ and $H'=T_2^0\cup\cdots\cup T_{p}^0=A_0$. If $T_p^0$ and $H$ are $(2)$-mergeable, then by Corollary~\ref{connect trimming lemma} we can remove some edges from $T_p^0$ to get an $((r-2)k+2,k)$-configuration inside $H\cup T_2$, a contradiction. Thus $T_p^0$ $\overline{1}2$-claims some pair $xy$ $1$-claimed by $H$. Let $D_1 \subseteq T_1^0$ and $D'_1 \subseteq T_p^0$ be two diamonds that are $(2|1)$-mergeable with $H$ and $H'$, respectively. If $|T_1^0| \ge 3$, then $$|T_2^0 \cup \cdots \cup T_{p-1}^0| \le k-4,$$ and hence$$|D_1 \cup T_2^0 \cup \cdots \cup T_{p-1}^0 \cup D'_1| \le k.$$ By Lemma~\ref{trimming lemma}, we may delete edges from $T_1^0$ and $T_p^0$ to obtain an $((r-2)k+2,k)$-configuration in $F'_0$, a contradiction. Therefore $|T_1^0| = 2$. Similarly, $|T_p^0| = 2$. Consequently, $$ |T_1^0 \cup \cdots \cup T_p^0| = k+1 \quad \text{and} \quad p=\ell. $$
Note that $S = D_1 = T_1^0 = \{e_1,e_2\}$. 

Finally, it suffices to prove that $|V(S) \cap V(H')| = 2$ and $|V(D'_1) \cap V(H)| = 2$.
Assume for a contradiction that $|V(S) \cap V(H')| \ge 3$. 
Then there exists, without loss of generality, $e_1$ such that $|e_1 \cap A_0| \geq 2.$ Since $|A_0| = k-1$ and $|V(A_0)| = (r-2)(k-1) + 2$, it follows that $|A_0 \cup \{e_1\}| = k$ and $|V(A_0) \cup V(e_1)| \le (r-2)k + 2,$ which contradicts the assumption that $G$ is $\mathcal{G}_k^{(r)}$-free. Similarly, we obtain $|D'_1 \cap H| = 2$. Hence $|S| = 2$ and $S \cup F$ satisfies Property $\mathbb{P}$.
\end{proof}

\begin{theorem}\label{structure of T with T>=k+1}
    Let $G$ be an $r$-uniform hypergraph on $n$ vertices that is $\mathcal{G}_k^{(r)}$-free.  
Let $F \in \mathcal{M}_2$ and suppose that $|F| \geq k+1$. Then $|F| \geq 2m(F) - k + 3.$
\end{theorem}
\begin{proof}
    %Since $|F|\geq m(F)+1\geq k$ and $G$ is $\mathcal{G}_k^{\,r}$-free, we have $|F|\geq k+1$. 
    If $m(F)\leq k-1$, the conclusion simply holds. We then assume that $m(F)\geq k$. Assume $F$ is obtained by merging $m$ $1$-clusters $T_1, T_2,\cdots, T_m$ in this order and let $s\in [m]$ be the smallest index such that $|\cup_{i=1}^sT_i|\geq k+1$. Let $H=T_1\cup\cdots\cup T_{s-1}$. If $T_s$ and $H$ are $(2)$-mergeable, then by Corollary~\ref{connect trimming lemma} we can remove some edges from $T_s$ to get an $((r-2)k+2,k)$-configuration inside $H\cup T_s$, a contradiction. Thus $T_s$ $\overline{1}2$-claims some pair $xy$ $1$-claimed by $H$. Let $D_0\subseteq T_s$ be the diamond $\overline{1}2$-claiming by $xy$. Let $T'_1=T_s$ and let $T'_2\in \{T_1,\cdots,T_{s-1}\}$ be a $1$-cluster $2$-mergeable with $T_s$. Let $(T'_2,\cdots,T'_{s})$ be the ordering of $\{T_1,\cdots,T_{s-1}\}$ returned by Lemma~\ref{trimming lemma}. Let $t\in [s]$ be the first index such that $|\cup_{i=1}^t T'_i|\geq k+1$. 
    
    By a similar proof as Theorem~\ref{we can add 2 edges each time}, we have that $D_0=T'_0$ and $D'_0:=T'_{s}$ are diamonds and $F_0=\cup_{i=1}^t T'_i$ satisfies that $|F_0|=k+1$. Moreover, $F_0$ has Property $\mathbb{P}$. If $F_0=F$, we are done. Otherwise there exists $T\in F\setminus F_0$ such that $T$ and $F_0$ are $1|2$-mergeable. By Theorem~\ref{we can add 2 edges each time}, we have $|T|=2$ and $F_1 = F_0 \cup T \subseteq F$ satisfies Property $\mathbb{P}$. Hence $|F_1| = |F_0| + 2$.

    Suppose we have obtained $F_i \subseteq F$ satisfying Property $\mathbb{P}$ with $|F_i| = |F_0| + 2i$. If $F_i = F$, then since $m(F_0) \le k-1$, we have $$ |F_i| \ge 2\bigl(m(F)-m(F_0)\bigr) + k + 1 \ge 2m(F) - k + 3, $$
and we are done. Otherwise, there exists $T' \in F \setminus F_i$ such that $F_i$ and $T'$ are $1|2$-mergeable. By Theorem~\ref{we can add 2 edges each time}, $|T'|=2$ and $F_{i+1} = F_i \cup T' \subseteq F$ also satisfies Property $\mathbb{P}$. This process terminates after finitely many steps, and therefore $|F| \ge 2m(F) - k + 3,$ as required.
\end{proof}

From the above proof, we obtain the following consequence.

\begin{theorem}\label{detail structure F>=k+1}
Let $F \in \mathcal{M}_2$ be obtained by merging $1$-clusters $T_1,\cdots,T_m$ such that $|F| \ge k+1$. Then $F$ can be obtained by $1$-clusters $ T'_1, \ldots, T'_\ell, T'_{\ell+1}, \ldots, T'_m $ merging in this order such that $\{T'_i\}_{i=1}^m$ is a re-ordering of $\{T_i\}_{i=1}^m$ and $H := \bigcup_{i=1}^{\ell} T'_i$ satisfies Property $\mathbb{P}$, where $T'_1, T'_\ell$ serve as the two flexible diamonds in the definition. Moreover, for every $i \ge \ell+1$, we have $|T'_i| = 2$, $T'_i$ is $(2|1)$-mergeable with $\bigcup_{j=1}^{i-1} T'_j$, and $\bigcup_{j=1}^{i} T'_j$ still satisfies Property $\mathbb{P}$.    
\end{theorem}

\begin{theorem}\label{Lower bound of P1 and P_12}
    For any $F\in \mathcal{M}_2$ with a composition $(e_1,\cdots,e_m)$, we have 
    \begin{equation}\label{P1 and P12 estimate}
    |P_{1}(F)|=\sum_{i\in [m]}\left(e_i\binom{r}{2}-e_i+1\right) \quad \text{and}\quad |P_{\overline{1}2}(F)|\geq 1-m+\sum_{i\in [m]}(e_i-1)(r-2)^2.
    \end{equation}
\end{theorem}
\begin{proof}
    Assume that $F$ is obtained by merging $T_1,\cdots,T_m$ in this order such that $|T_i|=e_i$ for $i\in [m]$. To prove the first equality, it suffices to show that for every $i \in [m]$, 
    \begin{equation}\label{P1 M1 size equation} |P_1(T_i)| = e_i\binom{r}{2} - e_i + 1,\end{equation}
 and that for every $i \in [m-1]$,
\begin{equation}\label{M1 P1 inter empty}P_1(T_{i+1}) \cap P_1\left(\bigcup_{j=1}^{i} T_j\right) = \emptyset. \end{equation}
Here, (\ref{M1 P1 inter empty}) holds because otherwise there exists $i \in [m-1]$ such that $T_{i+1}$ and some $T' \in \{T_1,\ldots,T_i\}$ would have been merged in $\mathcal{M}_1$. The following claim shows that (\ref{P1 M1 size equation}) holds.
    \begin{claim}\label{P1 M1 size equation Proof}
        Let $M\in \mathcal{M}'_1$ and $|M|=a$. Then $a\leq k-1$ and $|P_1(M)|=a\binom{r}{2}-a+1$.
    \end{claim}
    \begin{proof}[Proof of Claim~\ref{P1 M1 size equation Proof}]
    First, if $a \geq k+1$, then by Corollary~\ref{connect trimming lemma} we can delete some edges to obtain a $k$-configuration in $M$, a contradiction. Hence $a \leq k-1$.

    We prove the statement by induction on $a$. When $a=1$, the conclusion clearly holds. Now assume that for some $a \geq 2$, the statement holds for every partial $1$-cluster of size $a-1$. By the claim, consider the case $|M|=a$. By Corollary~\ref{connect trimming lemma}, there exists $M' \subseteq M$ that is also a partial $1$-cluster and satisfies $|M\setminus M'|=1$. Thus $|M'|=a-1$ and $|P_1(M')|=(a-1)\binom{r}{2}-a+2.$ It remains to show that $|P_1(M\setminus M')\cap P_1(M')|=1.$ It suffices to prove that $|V(M')\cap V(M\setminus M')|=2,$ which indeed holds. Otherwise, $$|V(M)|=|V(M')\cup V(M\setminus M')| \leq |V(M')|+r-3\leq (a-1)(r-2)+2+r-3 = (r-2)a+1, $$ contradicting the assumption that $G$ is $\mathcal{G}_k^{(r)}$-free since $a \le k - 1$.
    \end{proof}
    
    Now we prove the second inequality. It suffices to show that for every $i\in [m]$, \begin{equation}\label{ineq P-12 >= (e-1)(r-2)2}
        |P_{\overline{1}2}(T_i)|\geq (e_i-1)(r-2)^2
    \end{equation}
    and that for every $i\in [m-1]$, \begin{equation}\label{ineq P-12 cap P1 P2 = 1}
        |P_{\overline{1}2}(T_{i+1})\cap P_{\leq 2}\left(\bigcup_{j=1}^i T_j\right)|=1
    \end{equation}

    The inequality (\ref{ineq P-12 >= (e-1)(r-2)2}) holds because of the following.
    \begin{claim}\label{P12 M1 ineq proof}
        Let $M\in \mathcal{M}'_1$ and $|M|=a\leq k-1$, then $|P_{\overline{1}2}(M)|\geq (a-1)(r-2)^2.$
    \end{claim}
    \begin{proof}[Proof of Claim~\ref{P12 M1 ineq proof}]
        We prove the statement by induction on $a$. When $a=1$ or $2$, the conclusion clearly holds. Now assume that for some $a \geq 3$, the statement holds for every partial $1$-cluster of size $a-1$. Consider the case $|M|=a$. By Corollary~\ref{connect trimming lemma}, there exists a partial $1$-cluster $M' \subseteq M$ such that $|M'|=a-1$ and $|M\setminus M'|=1$. Let $M\setminus M'=\{e'\}$. Then $|V(e')\cap V(M')|=2$, and $e'$ together with some edge in $M'$ forms a diamond. Hence $|P_{\overline{1}2}(M)\setminus P_{\overline{1}2}(M')|\geq (r-2)^2.$ Therefore, $$|P_{\overline{1}2}(M)|=|P_{\overline{1}2}(M')|+|P_{\overline{1}2}(M)\setminus P_{\overline{1}2}(M')|\geq (a-2)(r-2)^2+(r-2)^2=(a-1)(r-2)^2.$$ %This completes the proof.
    \end{proof}
    It remains to prove that (\ref{ineq P-12 cap P1 P2 = 1}) holds. If $|F|\leq k-1$, then (\ref{ineq P-12 cap P1 P2 = 1}) holds since for each $i\in [m-1]$, $|V(T_{i+1})\cap (\cup_{j=1}^iT_j)|=2$. Now assume that $|F| \geq k+1$, and without loss of generality suppose that $T_1,\ldots,T_\ell,\ldots,T_m$ is the ordering of $1$-clusters given by Theorem~\ref{detail structure F>=k+1}. For $i\in[m]$, let $F_i=\bigcup_{j=1}^{i} T_j .$ It suffices to prove that for every $i\in [m-1]$, there is at most one pair  $xy\in P_{\overline{1}2}(T_{i+1})$ such that  $ xy \in P_1(F_i) \cup P_2(F_i). $

Note that when $i\le \ell-2$, the conclusion holds. Now assume that the conclusion holds for all $i\le j-1$ for some $j\ge \ell-1$, and consider the case $i=j$. Then $|F_j|\ge k-1$ and $|F_{j+1}|=2$. Furthermore, let $T_{j+1} = \{f_1,f_2\}$, and suppose that $T_{j+1}$ and  $T' \in \{T_1,\ldots,T_j\}$ are $(2|1)$-mergeable via the pair $xy$ and an edge $f \in T'$. For a contradiction, assume that there exist $x_0 \in f_1$ and $y_0 \in f_2$ with $\{x_0,y_0\} \neq \{x,y\}$ such that $x_0y_0 \in P_1(F_j) \cup P_2(F_j).$ 

If $x_0y_0 \in P_1(F_j)$, then there exists $T'_0 \in \{T_1,\ldots,T_j\} \setminus \{T'\}$ such that $T_{j+1}$ and $T'_0$ are $(2|1)$-mergeable via $x_0y_0$. By Lemma~\ref{trimming lemma} and the pair $T'_0 \subseteq F_j$, there exist $1$-clusters $T'_0, \ldots, T'_{k-1}$ such that $F_k$ is obtained by merging these $1$-clusters in this order. Let $t$ be the smallest index such that $|T'_0 \cup \cdots \cup T'_t| \geq k-1$ and denote $ H = T'_0 \cup \cdots \cup T'_{t-1},$ so that $|H| \leq k-2$ and $|V(H)|\leq (r-2)|H|+2$. First, $|H| \neq k-2$, because otherwise $|H \cup T_{j+1}| = k$ and $ |V(H) \cup V(T_{j+1})| \leq (r-2)k + 2,$ contradicting the assumption that $G$ is $\mathcal{G}_k^{(r)}$-free. Therefore, $|H| \le k-3$. 
\begin{case}
    If $T'\in \{T'_0,\cdots,T'_{t-1}\}$, then $|V(H)\cap V(T_{j+1})|\geq 3$, $|H\cup T_{j+1}|\leq k-1$. Hence $$|V(H)\cup V(T_{j+1})|\leq (r-2)|H\cup T_{j+1}|+1,$$ a contradiction. 
\end{case}
\begin{case}
    If $T'\notin \{T'_0,\cdots,T'_{t-1}\}$, then $|H|\leq k-4$. Otherwise if $|H|=k-3$, then $|H\cup T_{j+1}\cup \{f\}|=k$. Since $|V(f)\cap V(T_{j+1})|=2$ and $|V(T_{j+1})\cap H|=2$, we have that $$|V(H)\cup V(T_{j+1})\cup V(f)|\leq (r-2)k+2,$$ a contradiction. Therefore, $|H|\leq k-4$ and $|T'_{t}|\geq 3$. Let $H_1=H\cup T_{j+1}$ and $H_2=H\cup T_{j+1}\cup \{f\}$. Since $|H_1| \leq k-2$, $|H_2| \leq k-1$, and $$|V(H_1)| \le (r-2)|H_1| + 2, \quad |V(H_2)| \le (r-2)|H_2| + 2, $$  and moreover $$ |H_2 \cup T'_t| \ge |H_1 \cup T'_t| \ge k+1, $$ we can always remove edges from $T'_t$ to obtain a $k$-configuration inside $H_2\cup T'_t$, which is a contradiction.
\end{case}
If $x_0y_0 \in P_2(F_j)\setminus P_1(F_j)$, then there exists $T''_0 \in \{T_1,\ldots,T_j\}$ and a diamond $D \subseteq T''_0$ such that $x_0y_0 \in P_{\overline{1}2}(D)$. The remaining proof is similar to the case when $x_0y_0 \in P_1(F_j)$ discussed above.
\end{proof}

\subsubsection{Assigning weights}
Let $k\geq 4$ be an even positive integer and let $G$ be a $\mathcal{G}_k^{(r)}$-free $r$-graph with $n$ vertices. For a $2$-cluster $F$ and a pair $xy\in \binom{V(F)}{2}$, we define $$
w_F(xy) = \begin{cases}
1, & \text{if } 1 \in C_F(xy),\\[2mm]
\frac{2}{k-2}, & \text{if } 2 \in C_F(xy) \text{ and } 1 \notin C_F(xy),\\[1mm]
0, & \text{otherwise}. \end{cases}$$
And we define $$w(F):=\sum_{xy\in \binom{V(F)}{2}}w_{F}(xy), \quad \text{ for } F\in \mathcal{M}_2,$$ and $$w(xy):=\sum_{F\in \mathcal{M}_2,uv\in \binom{V(F)}{2}}w_{F}(uv), \quad \text{ for }xy\in\binom{V(G)}{2}.$$

We will need the following useful observation.
\begin{lemma}[\cite{Glock_Kim_Lichev_Pikhurko_Sun_2025}]\label{sum can not be k}
    For any $\mathcal{F}^{(r)}((r-2)k+2,k)$-free graph $G$, any $uv\in \binom{V(G)}{2}$ and any edge-disjoint subgraphs $F_1,\cdots,F_s\subseteq G$, the sum-set $\sum_{i=1}^s C_{F_i}(uv)=\{\sum_{i=1}^s m_i|m_i\in C_{F_i}(uv)\}$ does not contain $k$.
\end{lemma}

\begin{lemma}\label{w(xy)<=1}
    For every $xy\in \binom{V(G)}{2}$, it holds that $w(xy)\leq 1$.
\end{lemma}
\begin{proof}
    Given $uv\in \binom{V(G)}{2}$ and let $F_1,\cdots,F_s$ be all $2$-clusters assigning positive weight to $uv$, ordered so that $w_{F_1}(uv)\geq w_{F_2}(uv)\geq \cdots \geq w_{F_s}(uv)$. By the definition of weights, each $F_i$ satisfies that $1\in C_{F_i}(uv)$ or that $2\in C_{F_i}(uv)$ and $1\notin C_{F_i}(uv)$. 

    If there exists, say $F_1$, such that $1\in C_{F_1}(uv)$, then $w_{F_1}(uv)=1$. Moreover, for any $2$-clusters $F\in \mathcal{M}_2\setminus \{F_1\}$ we have $1\notin C_{F}(uv)$ and $2\notin C_{F}(uv)$, as otherwise we would merge them together by the merging rule of $\mathcal{M}_1$ and $\mathcal{M}_2$ respectively. Thus $w(uv)=w_{F_1}(uv)=1$.

    If for any $F\in \mathcal{M}_2$, $1\notin C_{F}(uv)$. Then by Lemma~\ref{sum can not be k}, there are at most $\frac{k}{2}-1$ $2$-clusters such that $2\in C_{F}(uv)$. Therefore, in this case, $s\leq \frac{k-2}{2}$ and hence $w(xy)\leq s\cdot \frac{2}{k-2}\leq 1$.
\end{proof}

\begin{lemma}\label{wF>=r choose 2 F}
    Let $k\geq 4$ and $r\geq 2+\sqrt{\frac{3}{2}k-4}$. For all $F\in \mathcal{M}_2$, we have $w(F)\geq \binom{r}{2}|F|$.
\end{lemma}
\begin{proof}
    Suppose that $F$ is obtained by merging $1$-clusters $T_1,\cdots,T_m\in \mathcal{M}_1$ with $|T_i|=e_i$. Then by Theorem~\ref{Lower bound of P1 and P_12}, we have $$|F|=\sum_{i\in [m]}e_i,\qquad |P_1(F)|=\sum_{i\in [m]}\left(e_i\binom{r}{2}-e_i+1\right)$$ and $$|P_{\overline{1}2}(F)|\geq 1-m+\sum_{i\in [m]}(e_i-1)(r-2)^2.$$ Therefore, $$\begin{aligned}
        w(F)&=\sum_{xy\in \binom{V(F)}{2}}w_F(xy)=|P_1(F)|+\frac{2}{k-2}|P_{\overline{1}2}(F)|\\&\geq \sum_{i\in [m]}\left(e_i\binom{r}{2}-e_i+1\right)+\frac{2}{k-2}(1-m)+\frac{2}{k-2}\sum_{i\in [m]}(e_i-1)(r-2)^2\\ &=\left(\frac{1}{2}r^2-\frac{1}{2}r+\frac{2}{k-2}(r-2)^2-1\right)|F|-\left(\frac{2}{k-2}r^2-\frac{8r}{k-2}+\frac{12-k}{k-2}\right)m+\frac{2}{k-2}\\&=\left(\frac{2}{k-2}r^2-\frac{8r}{k-2}+\frac{10-k}{k-2}+\frac{1}{2}r^2-\frac{1}{2}r\right)|F|-\left(\frac{2}{k-2}r^2-\frac{8r}{k-2}+\frac{12-k}{k-2}\right)m+\frac{2}{k-2}
    \end{aligned}$$
    To show that $w(F)\geq \binom{r}{2}|F|$, it suffices to prove that $$(2r^2-8r+10-k)|F|\geq (2r^2-8r+12-k)m-2,$$ i.e., $$|F|\geq m+\frac{m-1}{r^2-4r-\frac{k-10}{2}}.$$
    Since $r\geq 2+\sqrt{\frac{3}{2}k-4}$, we have $r^2-4r-\frac{k-10}{2}\geq k-3$. Note that $|F|\geq m+1$, if $m\leq k-2$ then we obtain that $$|F|\geq m+1\geq m+\frac{m-1}{k-3}\geq m+\frac{m-1}{r^2-4r-\frac{k-10}{2}}.$$ If $m\geq k-1$, then $|F|\geq m+1\geq k$. Since $G$ is $\mathcal{G}_k^{(r)}$-free, we have that $|F|\geq k+1$ and then by Theorem~\ref{structure of T with T>=k+1}, $|F|\geq 2m-k+3$. Thus, $$|F|\geq 2m-k+3\geq m+\frac{m-1}{k-3}\geq m+\frac{m-1}{r^2-4r-\frac{k-10}{2}}.$$
\end{proof}

Now we are ready to prove the upper bound of Theorem~\ref{main large r even k sqrtk}
\begin{proof}[Proof of the upper bound of Theorem~\ref{main large r even k sqrtk}]
    For a given $\mathcal{G}_k^{(r)}$-free $r$-graph $G$ with $n$ vertices and any $F \in \mathcal{M}_2$ with $xy \in \binom{V(F)}{2}$, define $w_F(xy)$ as above. Then $$|G|=\sum_{F\in \mathcal{M}_2}|F|\leq \binom{r}{2}^{-1}w(F)\leq \binom{r}{2}^{-1}\binom{{n}}{2},$$ which finishes the proof of the upper bound for $r\geq 2+\sqrt{\frac{3}{2}k-4}$.
\end{proof}

\section{Concluding remarks}    
We mainly study the Brown--Erd\H{o}s--S\'{o}s problem in the case when $k \ge 4$ is even. In particular, we are interested in the following problem.

\begin{problem}
Given an even integer $k \ge 4$, determine the smallest $r_0(k)$ such that 
$$\pi(r,k)=\frac{1}{r^2-r}, \text{ for all }r \ge r_0(k).$$
\end{problem}

%Inspired by the method of~\cite{glock20246-k4,Glock_Kim_Lichev_Pikhurko_Sun_2025,pikhurko2025quadratic-8-edges}, we prove that when $r \ge 2+\sqrt{\frac{3}{2}k-4}$, $$ \lim_{n\to\infty} f^{(r)}(n; (r-2)k+2,k)=\frac{1}{r^2-r}. $$
Our approach, inspired by the method of~\cite{glock20246-k4,Glock_Kim_Lichev_Pikhurko_Sun_2025,pikhurko2025quadratic-8-edges}, avoids dealing with delicate local structural details and shows that $r_0(k)$ can be chosen as $\Theta(k^{1/2})$, which improves a result of Letzter and Sgueglia~\cite{glock20246-k4} that $r_0(k)$ can be chosen as $\Theta(k^{3/2})$.

Moreover, note that when $k=4$, we show that $r_0(4)=4$, which agrees with the result in~\cite{glock20246-k4}. When $k\in\{6,8\}$, our result shows that $r_0(k) \le 5$. To show $r_0(k) = 4$ for $k\in\{6,8\}$, which is optimal, the authors of ~\cite{Glock_Kim_Lichev_Pikhurko_Sun_2025} and~\cite{pikhurko2025quadratic-8-edges} set up a more delicate weight assignment.

\section{Acknowledgment}

Yan Wang is supported by National Key R\&D Program of China under grant No. 2022YFA1006400, National Natural Science Foundation of China under grant No. 12571376 and SJTU-Warwick Joint Seed Fund.

%\newpage
\bibliographystyle{abbrv}
\bibliography{main}

@article{Glock_Kim_Lichev_Pikhurko_Sun_2025,
  title = {On the ($k + 2, k$)-problem of {Brown}, {Erd\H{o}s}, and {S\'{o}s} for $k = 5,6,7$},
  journal = {Canadian Journal of Mathematics},
  author = {Glock, Stefan and Kim, Jaehoon and Lichev, Lyuben and Pikhurko, Oleg and Sun, Shumin},
  year = {2025},
  pages = {1--43},
  doi = {10.4153/S0008414X25000021}
}

@article{letzter2025problemSIAM-k1.5,
  title={On a problem of {Brown}, {Erd\H{o}s}, and {S\'{o}s}},
  author={Letzter, Shoham and Sgueglia, Amedeo},
  journal={Proceedings of the American Mathematical Society},
  volume={153},
  number={07},
  pages={2729--2743},
  year={2025}
}

@article{pikhurko2025quadratic-8-edges,
title = {On the quadratic 8-edge case of the {Brown--Erd\H{o}s--S\'{o}s} problem},
journal = {European Journal of Combinatorics},
volume = {135},
pages = {104364},
year = {2026},
issn = {0195-6698},
doi = {https://doi.org/10.1016/j.ejc.2026.104364},
url = {https://www.sciencedirect.com/science/article/pii/S0195669826000326},
author = {Oleg Pikhurko and Shumin Sun}
}

@misc{brown1973some-BES-original,
  title={Some extremal problems on r-graphs, New directions in the theory of graphs (Proc. Third Ann Arbor Conf., Univ. Michigan, Ann Arbor, Mich, 1971)},
  author={Brown, WG and Erdos, P and S{\'o}s, VT},
  year={1973},
  publisher={Academic Press, New York}
}

@article{glock2019triple-pi33-BLMS,
  title={Triple systems with no three triples spanning at most five points},
  author={Glock, Stefan},
  journal={Bulletin of the London Mathematical Society},
  volume={51},
  number={2},
  pages={230--236},
  year={2019},
  publisher={Wiley Online Library}
}

@article{glock20246-k4,
  title={On the (6, 4)-problem of {Brown}, {Erd\H{o}s}, and {S\'{o}s}},
  author={Glock, Stefan and Joos, Felix and Kim, Jaehoon and K{\"u}hn, Marcus and Lichev, Lyuben and Pikhurko, Oleg},
  journal={Proceedings of the American Mathematical Society, Series B},
  volume={11},
  number={17},
  pages={173--186},
  year={2024}
}

@article{delcourt2024limit-PAMS-Resloved,
  title={The limit in the (k+2,k)-problem of {Brown}, {Erd\H{o}s}, and {S\'{o}s} exists for all k $\geq$ 2},
  author={Delcourt, Michelle and Postle, Luke},
  journal={Proceedings of the American Mathematical Society},
  volume={152},
  number={05},
  pages={1881--1891},
  year={2024}
}

@article{shangguan2023degenerate,
  title={Degenerate {Tur{\'a}n} densities of sparse hypergraphs {II}: a solution to the {Brown-Erd{\H{o}}s-S{\'o}s} problem for every uniformity},
  author={Shangguan, Chong},
  journal={SIAM Journal on Discrete Mathematics},
  volume={37},
  number={3},
  pages={1920--1929},
  year={2023},
  publisher={SIAM}
}

@article{rodl1985packing,
  title={On a packing and covering problem},
  author={R{\"o}dl, Vojt{\v{e}}ch},
  journal={European Journal of Combinatorics},
  volume={6},
  number={1},
  pages={69--78},
  year={1985},
  publisher={Elsevier}
}

@article{CONLON20231JCTB,
title = {A new bound for the {Brown–Erd\H{o}s–S\'{o}s} problem},
journal = {Journal of Combinatorial Theory, Series B},
volume = {158},
pages = {1-35},
year = {2023},
issn = {0095-8956},
doi = {https://doi.org/10.1016/j.jctb.2022.08.005},
url = {https://www.sciencedirect.com/science/article/pii/S0095895622000818},
author = {David Conlon and Lior Gishboliner and Yevgeny Levanzov and Asaf Shapira}
}

@article{Erdos_Frankl_Rodl_1986,
  author  = {Erd\H{o}s, P. and Frankl, P. and R\"{o}dl, V.},
  title   = {The asymptotic number of graphs not containing a fixed subgraph and a problem for hypergraphs having no exponent},
  journal = {Graphs and Combinatorics},
  volume  = {2},
  number  = {1},
  pages   = {113--121},
  year    = {1986}
}

@article{Ruzsa_Szemeredi_1978,
  author  = {Ruzsa, Imre Z. and Szemer\'{e}di, Endre},
  title   = {Triple systems with no six points carrying three triangles},
  journal = {Colloquia Mathematica Societatis J\'{a}nos Bolyai},
  volume  = {18},
  pages   = {939--945},
  year    = {1978},
  note    = {Combinatorics (Keszthely, 1976)}
}

@article{Conlon_Fox_2013_survey,
  author  = {Conlon, D. and Fox, J.},
  title   = {Graph removal lemmas},
  journal = {Surveys in Combinatorics},
  volume  = {409},
  pages   = {1--49},
  year    = {2013}
}

@article{Roth_1953_3ap,
  author  = {Roth, K. F.},
  title   = {On certain sets of integers},
  journal = {Journal of the London Mathematical Society},
  volume  = {28},
  number  = {1},
  pages   = {104--109},
  year    = {1953}
}

@article{Gowers_2007,
  author  = {Gowers, W. T.},
  title   = {Hypergraph regularity and the multidimensional Szemer\'{e}di theorem},
  journal = {Annals of Mathematics},
  volume  = {166},
  pages   = {897--946},
  year    = {2007}
}

@article{Nagle_Rodl_Schacht_2006,
  author  = {Nagle, B. and R\"{o}dl, V. and Schacht, M.},
  title   = {The counting lemma for regular k-uniform hypergraphs},
  journal = {Random Structures \& Algorithms},
  volume  = {28},
  number  = {2},
  pages   = {113--179},
  year    = {2006}
}

@article{Rodl_Skokan_2004_regular,
  author  = {R\"{o}dl, V. and Skokan, J.},
  title   = {Regularity lemma for k-uniform hypergraphs},
  journal = {Random Structures \& Algorithms},
  volume  = {25},
  number  = {1},
  pages   = {1--42},
  year    = {2004}
}

@article{Rodl_Skokan_2006_apply,
  author  = {R\"{o}dl, V. and Skokan, J.},
  title   = {Applications of the regularity lemma for uniform hypergraphs},
  journal = {Random Structures \& Algorithms},
  volume  = {28},
  number  = {2},
  pages   = {180--194},
  year    = {2006}
}

@article{Sarkozy_Selkow_2004_k+d,
  author  = {S\'{a}rk\"{o}zy, G. N. and Selkow, S.},
  title   = {An extension of the {Ruzsa-Szemer\'{e}di theorem}},
  journal = {Combinatorica},
  volume  = {25},
  number  = {1},
  pages   = {77--84},
  year    = {2004}
}

@article{Solymosi_Solymosi_2017_small_k,
  author  = {Solymosi, D. and Solymosi, J.},
  title   = {Small cores in 3-uniform hypergraphs},
  journal = {Journal of Combinatorial Theory, Series B},
  volume  = {122},
  pages   = {897--910},
  year    = {2017}
}

@article{shapira2021ramsey,
  title={A Ramsey variant of the {Brown--Erd{\H{o}}s--S{\'o}s conjecture}},
  author={Shapira, Asaf and Tyomkyn, Mykhaylo},
  journal={Bulletin of the London Mathematical Society},
  volume={53},
  number={5},
  pages={1453--1469},
  year={2021},
  publisher={Wiley Online Library}
}

@article{Alon_Shapira_2006,
  author  = {Alon, N. and Shapira, A.},
  title   = {On an extremal hypergraph problem of {Brown, Erd\H{o}s and S\'{o}s}},
  journal = {Combinatorica},
  volume  = {26},
  number  = {6},
  pages   = {627--646},
  year    = {2006}
}
 
\end{document}